\documentclass[a4paper,12pt]{amsart}

\usepackage{eucal,fullpage,times,amsmath,amsthm,amssymb,mathrsfs,stmaryrd,color,enumerate,accents}
\usepackage[all]{xy}
\usepackage{url}
\usepackage[leqno]{amsmath}
\usepackage{amsthm}
\usepackage{amsfonts}
\usepackage{amssymb}
\usepackage{eucal}
\usepackage[all]{xy}
\CompileMatrices

\usepackage{mathrsfs}

\usepackage{amsmath,amssymb,mathrsfs,xy,amsthm,bm,url}
\input{xy}
\xyoption{all}

\usepackage{xspace}
 \usepackage[ansinew]{inputenc}
 \usepackage[T1]{fontenc}
 \usepackage{hyperref}

\usepackage{hyperref}

\theoremstyle{definition}
  \newtheorem{definition}[subsection]{Definition}
  \newtheorem{notation}[subsection]{Notation}
  \newtheorem{definition-proposition}[subsection]{Definition-Proposition}
  
  \newtheorem{remark}[subsection]{Remark}

\theoremstyle{theorem}
  \newtheorem{theorem}[subsection]{Theorem}
  \newtheorem{proposition}[subsection]{Proposition}
  \newtheorem*{proposition*}{Proposition}
  \newtheorem{lemma}[subsection]{Lemma}
  \newtheorem{corollary}[subsection]{Corollary}

\newcommand{\maxx}{{\mathrm{max}}}

\newcommand{\Cbb}{\mathbb{C}}

\newcommand{\Nbb}{\mathbb{N}}
\newcommand{\Qbb}{\mathbb{Q}}
\newcommand{\Rbb}{\mathbb{R}}
\newcommand{\Sbb}{\mathbb{S}}
\newcommand{\Zbb}{\mathbb{Z}}

\newcommand{\Cbf}{\mathbf{C}}

\newcommand{\Gbf}{\mathbf{G}}

\newcommand{\Pbf}{\mathbf{P}}

\newcommand{\Tbf}{\mathbf{T}}
\newcommand{\Ubf}{\mathbf{U}}
\newcommand{\Vbf}{\mathbf{V}}
\newcommand{\Wbf}{\mathbf{W}}

\newcommand{\Gfrak}{{\mathfrak{G}}}

\newcommand{\Xfrak}{\mathfrak{X}}

\newcommand{\Xrm}{\mathrm{X}}

\newcommand{\Hscr}{\mathscr{H}}

\newcommand{\der}{\mathrm{der}}

\newcommand{\ra}{\rightarrow}
\newcommand{\lra}{\longrightarrow}

\newcommand{\mono}{\hookrightarrow}
	
\newcommand{\isom}{\cong}

\newcommand{\Hom}{\mathrm{Hom}}

\newcommand{\GSp}{\mathrm{GSp}}
\newcommand{\Sp}{\mathrm{Sp}}
\newcommand{\Ker}{\mathrm{Ker}}

\newcommand{\Nm}{\mathrm{Nm}}

\newcommand{\GL}{{\mathbf{GL}}}

\newcommand{\inv}{{-1}}
\newcommand{\ot}{\overset}

\newcommand{\wrt}{{with\ respect\ to}\xspace}
\newcommand{\ifof}{{if\ and\ only\ if}\xspace}
\newcommand{\cosg}{{compact\ open\ subgroup}\xspace}

\newcommand{\Qac}{\bar{\mathbb{Q}}}
\newcommand{\Gal}{\mathrm{Gal}}

\newcommand{\pibar}{\overline{\pi}_\circ}

\newcommand{\mult}{{\mathbb{G}_\mathrm{m}}}
\newcommand{\adele}{{\hat{\mathbb{Q}}}}

\newcommand{\Qbbp}{{\mathbb{Q}_p}}
\newcommand{\Zbbp}{{\mathbb{Z}_p}}

\newcommand{\TW}{{{(\mathbf{T},w)}}}

\newcommand{\ord}{\mathrm{ord}}

\newcommand{\rec}{{\mathrm{rec}}}
\newcommand{\Aut}{{\mathrm{Aut}}}
\title{bounded equidistribution of special subvarieties II}

\author{Ke Chen}
\address{Department of Mathematics, University of Science and Technology of China, 230026 Hefei, Anhui Province, China}

\email{kechen@ustc.edu.cn}



\begin{document}

\maketitle

\tableofcontents

\section*{introduction}

This paper is a continuation of \ref{chen bounded 1}. In \ref{chen bounded 1} we have studied the equidistribution of bounded sequences of special subvarieties in an arbitrary mixed Shimura variety. In this paper we prove a lower bound formula for pure special subvarieties in a mixed Shimura variety, define the notion of test invariant of a special subvariety which is not necessarily pure, and prove the equidistribution of sequences of special subvarieties with bounded test invariants. Similar to the pure case treated in \cite{ullmo yafaev}, the lower bound estimation and the equidistribution are established under the GRH for CM fields.

The main results of this paper are:

\begin{theorem}Assume GRH for CM fields.
Let $M=M_K(\Pbf,Y)$ be a mixed Shimura variety, with $K\subset\Pbf(\adele)$ a \cosg of fine product type, and $E$ its reflex field. Fix an integer $N$, Then for $M'$ a pure special subvariety in $M$ defined by $(w\Gbf'w^\inv, wX;wX^+)$, we have $$\#\Gal(\Qac/E)\cdot M'\geq c_ND_N(\Tbf)\prod_{p\in\Delta(\Tbf,K_\Gbf(w))}\maxx\{1,I(\Tbf,K_\Gbf(w)_p)\}$$ where \begin{itemize}

\item $c_N$ is some absolute constant, independent of $K$, $M'$;

\item $\Tbf$ is the connected center of $\Gbf'$, $D_N(\Tbf):=(\log (D(\Tbf)))^N$ with $D(\Tbf)$ the absolute discriminant of the splitting field of $\Tbf$;

\item $K_\Gbf(w)=\{g\in K_\Gbf:wgw^\inv g^\inv\in K_\Wbf$ following the notations in \cite{chen bounded 1}, and $\Delta(\Tbf,K_\Gbf(w))$ is the set of rational primes such that $\Tbf(\Qbbp)\cap K_\Gbf(w)_p\subsetneq K_{\Tbf,p}^\maxx$, $K_{\Tbf,p}^\maxx$ being the maximal \cosg of $\Tbf(\Qbbp)$;

\item $I(\Tbf,K_\Gbf(w)_p)=b\cdot[K_{\Tbf,p}^\maxx:\Tbf(\Qbbp)\cap K_\Gbf(w)_p]$ with $b$ some absolute constant independent of $K$, $M'$.

\end{itemize}

\end{theorem}

For a general special subvariety $M'\subset M$ which is not pure, we introduce the notion of test invariants $\tau_M(M')$ as a substitute for the lower bound of Galois orbits, and we get

\begin{theorem}Assume GRH for CM fields. Let $M$ be a mixed Shimura variety defined by $(\Pbf,Y)$ at some level $K$ of fine product type. Let $(M_n)$ be a sequence of special subvarieties in $M$, such that the sequence of test invariants $(\tau_M(M_n))$ is bounded. Then the sequence $(M_n)$ is bounded by some finite bounding set $B$ in the sense of \cite{chen bounded 1}. In particular, the Zariski closure of $\bigcup_nM_n$ is a finite union of special subvarieties bounded by $B$.

\end{theorem}

Note that \cite{ullmo yafaev} has formulated their main lower bound via the intersection degree against the automorphic line bundle on pure Shimura varieties, which fits into the framework of \cite{klingler yafaev}. We do not need this step yet in this paper, and we stick to the counting of Galois orbits.

\section{lower bound of the Galois orbit of a special subvariety}

In the pure case, Ullmo and Yafaev proved the following lower bound of Galois orbits of special subvarieties in a pure Shimura variety:

\begin{theorem}[lower bound in the pure case, cf. \cite{ullmo yafaev} Theorem 2.19]\label{lower bound in the pure case} Let $S=M_K(\Gbf,X)$ be a pure Shimura variety with reflex field $E$, with $K\subset\Gbf(\adele)$ a level of  product type. Assume the GRH for CM fields, and fix an integer $N>0$. Then for $S'\subset S$ a $\Tbf$-special subvariety, we have $$\#\Gal(\Qac/E)\cdot S'\geq c_N\cdot D_N(\Tbf)\cdot \prod_{p\in \Delta(\Tbf,K)} \maxx\{1, I(\Tbf,K_p)\}$$ with
\begin{itemize}

\item $D_N(\Tbf)=(\log D(\Tbf))^N$, where $D(\Tbf)$ is the absolute discriminant of the splitting field of the $\Qbb$-torus $\Tbf$;

\item $\Delta(\Tbf,K)$ is the set of rational primes $p$ such that $K_{\Tbf,p}\subsetneq K_{\Tbf,p}^\maxx$, where \begin{itemize}
\item $K_{\Tbf,p}=\Tbf(\Qbbp)\cap K_p$,
\item $K_{\Tbf,p}^\maxx$ the unique maximal \cosg of $\Tbf(\Qbbp)$

\end{itemize} $\Delta(\Tbf,K)$  is finite, i.e. $K_{\Tbf,p}=K_{\Tbf,p}^\maxx$ for all but finitely many $p$'s.

\item $I(\Tbf,K_p)=b[K_{\Tbf,p}^\maxx :K_{\Tbf,p}]$


\item and $c_N,b\in\Rbb_{>0}$ are constants independent of $K$, $\Tbf$.
\end{itemize}
\end{theorem}

\begin{remark}[dependence on levels]\label{dependent on levels}
(1) The results in \cite{ullmo yafaev} was formulated for an ambient pure Shimura variety $M_K(\Gbf,X)$ with $\Gbf$ semi-simple of adjoint type. The $\Qbb$-torus that appear as the connected centers are compact. In our case, using the condition \ref{fibred mixed shimura data}(e), and taking quotient by the connected center $\Cbf$ of $\Gbf$, we can reduce to the setting of \cite{ullmo yafaev}, as the contribution of Galois orbits from $\Cbf$ can be removed when we pass from $E$ to some number field splitting $\Cbf$.

(2)  The estimation depends on an embedding of $(\Gbf,X)$ into some ambient pure Shimura datum $(\Gfrak,\Xfrak)$, and a faithful algebraic representation $\rho:\Gfrak\ra\GL_{n\Qbb}$.

  The constants $c_N$ and $b$ are independent of $K$. This was not mentioned explicitly in \cite{ullmo yafaev}, but one can verify through their arguments that $c_N$ and $b$ are determined by $(\Gfrak,\Xfrak)$ and $\rho$. $c_N$ does depend on the prescribed integer $N$, but any fixed $N$ will suffice.

(3) The quantity $D_N(\Tbf)$ is independent of $K$, while $I(\Tbf,K_p)$ describe the position of $\Tbf(\adele)$ relative to $K_p$. Whether $p$ lies in $\Delta(\Tbf,K)$ or not is closely related to the integral structure of $\Tbf$ at $p$, and is also related to the reduction property of the special subvarieties. See \cite{edixhoven yafaev} and \cite{yafaev duke} for details.

(4) The estimation in \cite{ullmo yafaev} was formulated using intersection degrees against the ample line bundle of top degree automorphic forms on $S=M_K(\Gbf,X)$. Actually the intersection degree of a single (connected) special subvariety only contributes as a real number greater than 1 in the lower bound. The formulation is used in further study of unbounded orbits in \cite{klingler yafaev}.

\end{remark}

We can thus consider the lower bound of the Galois orbits of pure special subvariety in a given mixed Shimura variety.

\begin{theorem}[orbit of a pure special subvariety]\label{orbit of a pure special subvariety} Let $(\Pbf,Y)=(\Ubf,\Vbf)\rtimes(\Gbf,X)$ be a mixed Shimura subdatum, defining a mixed Shimura variety $M$ at a level $K$ of  fine product type. Write $E$ for the reflex field of $(\Pbf,Y)$, and $\pi$ for the natural projection $M\ra S=M_{K_\Gbf}(\Gbf,X)$ wite $\iota(0)$ the zero section.

Let $M'$ be a pure special subvariety of $M$ defined by a subdatum of the form $(w\Gbf'w^\inv,wX')$ for some pure subdatum $(\Gbf',X')\subset(\Gbf,X)$ and $w\in\Wbf(\Qbb)$. Then we have the following lower bound assuming the GRH for CM fields, using the same constants $c_N$, $b$, and notations in \ref{lower bound in the pure case}:
$$\Gal(\Qac/E)\cdot M'\geq c_ND_N(\Tbf)\prod_{p\in\Delta(\Tbf,K_\Gbf(w))}\maxx\{1, I(\Tbf, K_\Gbf(w)_p)\}$$ where \begin{itemize}
  \item $\Tbf$ is the connected center of $\Gbf'$, and $D_N(\Tbf)=(\log D(\Tbf))^N$;  \item $K_\Gbf(w)$ is the subgroup $$\{g\in K_\Gbf:wgw^\inv g^\inv\in K_\Wbf\}$$ and $I(\Tbf,K_\Gbf(w)_p)=b[K_{\Tbf,p}^\maxx : K_\Tbf(w)_p]$ with $K_\Tbf(w)=K_\Gbf(w)\cap\Tbf(\adele)$ .\end{itemize}
\end{theorem}

Before entering the proof, we first justify some notions in the statement of the theorem:

\begin{lemma}
  In the statement \ref{pure special subvarieties in a mixed Shimura variety} above,  $K_\Gbf(w)_p=\Gbf(\Qbbp)\cap K_\Gbf(w)$ is a \cosg in $K_{\Gbf,p}$, and the inclusion $K_\Gbf(w)_p\subset K_{\Gbf,p}$ is an equality for all but finitely many rational primes $p$'s. In particular, the group $K_\Gbf(w)$ is a \cosg in $K_\Gbf$ of fine product type, i.e. $K_\Gbf(w)=\prod_pK_\Gbf(w)_p$.
\end{lemma}

\begin{proof}

For all but finitely many $p$'s, we have $w\in\Wbf(\Qbb)$ lies in $K_{\Wbf,p}$ and $wgw^\inv g^\inv\in K_{\Wbf,p}$ for $g\in K_{\Gbf,p}$.

When $w\notin K_{\Wbf,p}$, write $w=(u,v)$ for some $u\in\Ubf(\Qbb)$ and $v\in\Vbf(\Qbb)$, then by \ref{group law} we have $w^n=(nu,nv)$ for any $n\in\Zbb$, hence the subgroup $K'_{\Wbf,p}$ generated by $w$ and $K_{\Wbf,p}$ is compact and open, containing $K_{\Wbf,p}$ as a subgroup of finite index. $K_{\Gbf,p}$ stabilizes $K_{\Wbf,p}$, hence a \cosg in $K_{\Gbf,p}$ of finite index stabilizes $K'_{\Wbf,p}$.
\end{proof}

The theorem is reduced to the following

\begin{lemma}[reflex field]\label{reflex field} $(\Pbf,Y)=\Wbf\rtimes(\Gbf,X)$ has the same reflex field as $(\Gbf,X)$ does, and the action of $\Gal(\Qac/E)$ on $M_{K^w_\Gbf}(w\Gbf w^\inv, wX)$ is identified with its action on $M_{K_\Gbf(w)}(\Gbf,X)$ where $K^w_\Gbf:=w\Gbf(\adele)w^\inv\cap K_\Wbf\rtimes K_\Gbf$.

\end{lemma}

\begin{proof}

From the definition of reflex fields \cite{pink thesis} Chapter 11, we know that for a morphism of mixed Shimura data $(\Pbf,Y)\ra(\Pbf',Y')$ we have $E(\Pbf,Y)\supset E(\Gbf,X)$. Thus $E(\Pbf,Y)=E(\Gbf,X)$ because we have the natural projection and the zero section. Conjugation by $w\in \Wbf(\Qbb)$ gives the isomorphism $(\Gbf,X)\isom (w\Gbf w^\inv, wX)$ as maximal pure subdata of $(\Pbf,Y)$. It also induces an isomorphism of pure Shimura variaties $M_{K_\Gbf}(\Gbf,X)\isom M_{w K_\Gbf w^\inv}(w\Gbf w^\inv, wX)$.

It is easy to verify that $wK_\Gbf w^\inv\cap K_\Wbf\rtimes K_\Gbf=wK_\Gbf(w)w^\inv$. Therefore when $K_\Gbf=K_\Gbf(w)$, the conjugation by $w$ gives an isomorphism $M_{K\cap w\Gbf(\adele)w^\inv}(w\Gbf w^\inv, wX)\isom M_{K_\Gbf}(\Gbf,X)$. In fact,  the Hecke translate $M_{wKw^\inv}(\Pbf,Y)\isom M_{K}(\Pbf,Y)$ sends the zero section $M_{wK_\Gbf w^\inv}(w\Gbf w^\inv, wX)$ \wrt $(\Pbf,Y)=\Wbf\rtimes(w\Gbf w^\inv,wX)$ to the zero section $M_{K_\Gbf}(\Gbf,X)$ \wrt $(\Pbf,Y)=\Wbf\rtimes(\Gbf,X)$. In particular, for any pure subdatum $(\Gbf',X')$ of $(\Gbf,X)$, the special subvariety defined by $(w\Gbf' w^\inv, wX';wX'^+)$ is isomorphic to the one defined by $(\Gbf',X';X'^+)$, and the isomorphism respects the canonical models.

Hence the theorem holds trivially when $K_\Gbf=K_\Gbf(w)$. When $K_\Gbf(w)\subsetneq K_\Gbf$, it suffices to take a base change $f: S'=M_{K_\Gbf(w)}(\Gbf,X)\ra S=M_{K_\Gbf}(\Gbf,X)$ which is a morphism of pure Shimura varieties defined over $E(\Gbf,X)$. The base change is finite \'etale as we have taken $K_\Gbf$ to be neat. It respects the special subvarieties of $M=M_K(\Pbf,Y)$ and of $M_{S'}$ as well as their canonical models, hence the lemma.
\end{proof}

Before we take a closer look at the term $I(\Tbf,K_\Gbf(w)_p)$, we introduce the following
\begin{notation}\label{order}
For $w=(u,v)\in\Wbf(\Qbb)$, we have $(u,v)^n=(nu,nv)$, hence it makes sense to talk about the order of $w=(u,v)$ \wrt $K_\Wbf$: it is the smalles positive integer $m>0$ such that $w^m\in K_\Wbf$, i.e. $mu\in K_\Ubf$ and $mv\in K_\Vbf$, which makes sense because $u$ and $v$ are in $\Ubf(\Qbb)$ and $\Vbf(\Qbb)$ respectively. We can also talk about the $p$-order of $w$ \wrt $K_\Wbf$, namely the integer $m\in\Nbb$ such that $w^n\in K_{\Wbf,p}$ \ifof $p^m$ divides $n$.
\end{notation}

In the lower bound we have the set $\Delta(\Tbf,K_\Gbf(w))$ containing the subset $\delta(\Tbf,K_\Gbf(w))$ of primes $p$ such that $K_{\Tbf,p}\supsetneq K_\Tbf(w)_p$. We want to show that for $p\in\delta(\Tbf,K_\Gbf(w))$, the inequality $$[K_{\Tbf,p}^\maxx:K_\Tbf(w)p]\geq c\cdot\ord_p(w,K_\Wbf)$$ holds for some absolute constant $c$ independent of $K,w,\Tbf$. This is clear when $\Tbf\isom\mult$ acts on $\Ubf$ and $\Vbf$ by scaling $g(u,v)=(gu,gv)$ using the central cocharacters $\mult\ra\GL_\Ubf$ and $\mult\ra\GL_\Vbf$. In fact, for the action on $\Ubf(\Qbbp)$, $K_{\Tbf,p}$ is a \cosg of $\mult(\Qbbp)=\Qbbp$ stabilizing $K_{\Ubf,p}$, contained in the maximal \cosg $K_{\Tbf,p}^\maxx\isom\Zbbp^\times$, and $K_\Tbf(u)_p$ is the stabilizer of the class $u$ modulo $K_{\Ubf,p}$. Since the automorphism by $K_{\Tbf,p}^\maxx\isom\Zbbp^\times$ preserves the torsion order in $\Ubf(\Qbbp)/K_{\Ubf,p}$ and leaves the line $\Qbbp u$ stable, we see that $$[K_{\Tbf,p}^\maxx:K_\Tbf(u)_p]\geq (p-1)p^{m-1}$$ because $(p-1)p^{m-1}$ is the number of elements of order $p^m$ in $\Qbbp u$ modulo $K_{\Ubf,p}$. The case of $\Tbf$ acting on $\Vbf$ is similar under our assumption $\Tbf\isom\mult$, and it is obvious that $$\ord_p(w,K_\Wbf)=\maxx\{\ord_p(u,K_\Ubf), \ord_p(v,K_\Vbf)\}$$ hence $$[K^\maxx_{\Tbf,p}:K_\Tbf(w)_p]\geq (1-\frac{1}{p})\ord_p(w,K_\Wbf)\geq \frac{1}{2}\ord_p(w,K_\Wbf).$$
\bigskip

In general, the $\Qbb$-torus does admit quotients isomorphic to split $\Qbb$-tori:

\begin{lemma}[split tori]\label{split tori}Let $(\Pbf,Y)=(\Ubf,\Vbf)\rtimes(\Gbf,X)$ be a mixed Shimura datum, with $\Tbf$ the connected center of $\Gbf$. Then for the actions of $\Tbf$ on $\Ubf$ and on $\Vbf$, 

(1) there is a $\Tbf$-equivariant decomposition $\Ubf=\oplus_{i=1,\cdots,r}\Ubf_i$ such that $\Tbf$ acts on $\Ubf_i$ via the central scaling $\mult\ra\GL_{\Ubf_i}$;

(2) there exists a $\Tbf$-equivariant decomposition $\Vbf=\oplus_{j=1\cdots,s}\Vbf_j$ such that  in the representation $\Tbf\ra\GL_{\Vbf_j}$, the image of $\Tbf$ contains the center of $\GL_{\Vbf_j}$.

\end{lemma}

\begin{proof}

$\Gbf$ and $\Tbf$ being reductive, it suffices to consider the case when $\Ubf$ and $\Vbf$ are irreducible as representations of $\Gbf$.

(1) This is clear because by \cite{pink thesis} 2.16, $\Gbf$, hence $\Tbf$,  acts on $\Ubf$ through a split $\Qbb$-torus, hence the irreducible representation $\Ubf$ is one-dimensional. Since $\Ubf$ is of Hodge type $(-1,-1)$, the action of $\Gbf$, hence the action of $\Tbf$ on it is through the central scaling.

(2) $\rho:\Gbf\ra\GL_\Vbf$ is an irreducible representation of $\Gbf$, such that for any $x\in X$, the composition $\rho\circ x$ is a rational Hodge structure of type $\{(-1,0),(0,-1)\}$. It thus follows from the definition of pure Shimura data \cite{pink thesis} 1. ?.?  that the Hodge structure is polarizable, namely $\Gbf$ preserves some polarization $\psi:\Vbf\times\Vbf\ra\Qbb(-1)$ up to scalars, hence the representation factors through the Siegel datum, i.e. $(\Gbf,X)\ra(\GSp_\Vbf,\Hscr_\Vbf)$.

It suffices to show that the image $\Gbf\ra\GSp_\Vbf$ contains the center of $\GSp_\Vbf$. If it does not contains the center, then it is contained in   $\Sp_\Vbf=\Ker(\GSp_\Vbf\ot{\det}\lra\mult)$. The construction in \cite{ullmo yafaev} ?.? shows that $(\Sp_\Vbf,X')$ is a pure Shimura subdatum of $(\GSp_\Vbf,\Hscr_\Vbf)$ with $X'$ the $\Sp_\Vbf(\Rbb)$-orbit of the image of $X$ in $\Hscr_\Vbf$, which is ridiculous because $x(\Sbb)\notin\Sp_{\Vbf,\Rbb}$ for any $x\in \Hscr_\Vbf$ due to the $\Rbb$-torus $\mult_\Rbb\subset\Sbb$.
\end{proof}

We are thus led to 

\begin{proposition}[torsion order]\label{torsion order}

For $p\in\delta(\Tbf,K_\Gbf(w))$ as above, we have $[K_{\Tbf,p}^\maxx:K_\Tbf(w)_p]\geq cp^{\ord_p(w,K_\Wbf)}$ for some constant $c$ which is independent of $K,w,\Tbf$.

\end{proposition}

\begin{proof}

Since the number of irreducible representations in $\Ubf$ (resp. in $\Vbf$) is uniformly bounded by the dimension of $\Ubf$ (resp. of $\Vbf$), we are reduced to the case when $\Ubf$ and $\Vbf$ are irreducible under $\Gbf'$.

Thus $\Ubf$ is one-dimensional, with $\Tbf$ acts on it through the central scaling $\mult\isom\GL_\Ubf$. It suffices to show that for each prime $p$, the homomorphism $\Tbf(\Qbbp)\ra\mult(\Qbbp)$ sends $K_{\Tbf,p}^\maxx$ to a \cosg of $\Zbbp^\times$ whose index in $\Zbbp^\times$ is uniformly bounded.

Recall that the splitting field $F=F_\Tbf$ is a number field, and $[F:\Qbb]$ is uniformly bounded by some constant $c_1$ that only depends on the dimension of $\Gbf$; in particular, we can rearrange $c_1\in\Nbb_{>0}$ such that for any connected center $\Tbf$ of pure Shimura subdatum $(\Gbf',X')$ of $(\Gbf,X)$, $[F_\Tbf:\Qbb]$ divides $c_1$ and $[F_\Tbf(p):\Qbbp]$ divides $c_1$ with $F_\Tbf(p)$ the splitting field of the $\Qbbp$-torus $\Tbf_{\Qbbp}$.

Fix $p$ a prime, $F$ the splitting field of $\Tbf_\Qbbp$ over $\Qbbp$. Write $\Xrm$ for the group of characters $\Hom(\Tbf_F,\mult_F)$ with the natural action of $\Gamma=\Gal(F/\Qbbp)$. Then $\Tbf(F)=\Hom(\Xrm,F^\times)$, and $\Tbf(\Qbbp)=\Hom_\Gamma(\Xrm,F^\times)$ is the $\Gamma$-fixed part of $\Tbf(F)$. The maximal \cosg of $\Tbf(F)$ is $\Hom(\Xrm_\Tbf,O_F^\times)$ with $O_F$ the integer ring of $F$, and the norm map $\Nm:\Tbf(F)\ra\Tbf(\Qbbp)$ sends $\Hom(\Xrm,O_F^\times)$ into the maximal \cosg $K_{\Tbf,p}^\maxx$ of $\Tbf(\Qbbp)$.

From \ref{split tori} we have homomorphism of $\Qbb$-tori $\Tbf\ra\mult^d$, where $\mult^d$ is a $\Qbb$-subtorus in $\GL_\Ubf$ (or $\GL_\Vbf$) that acts as central scaling on direct summands of $\Ubf$ (or on $\Vbf$). The condition on Hodge types show that for the corresponding map of characters $\Zbb^d\ra \Xrm$, the image of $\Zbb^d$ is of index at most 2 in a direct summand of $\Xrm$ on which $\Gamma$ acts trivially. 

Now consider the commutative diagram $$\xymatrix{\Hom(\Xrm,O_F^\times)\ar[d]^{\Nm}\ar[r] & \Hom(\Zbb^d,O_F^\times)\ar[d]^{\Nm}\\ K_{\Tbf,p}^\maxx\ar[r] & \Hom(\Zbb^d,\Zbbp^\times)}$$ where the horizontal maps are induced from the norm $\Nm:F^\times\ra\Qbbp^\times$, the horizontal maps are induced by $\Tbf\ra\mult^d$, and the image of the upper horizontal map is of index at most $2^d$. Since the degree $[\Zbbp^\times:\Nm(O_F^\times)]\leq [F:\Qbbp]$ by local class field theory, we see that the image of the lower horizontal map $K_{\Tbf,p}^\maxx\ra\Hom(\Zbb^r,\Zbbp^\times)$ is of finite index, and the index is bounded by a constant that only depends on $c_1$ and the dimension of $\Tbf$.\end{proof}

We summarize the above computation into the following
\begin{corollary}[unipotent defects]\label{unipotent defects}
There is some constant $c$, independent of $c,K,w$, such that in the expression $I(\Tbf,K_\Gbf(w)_p)$ in \ref{orbit of a pure special subvariety}, we have  $I(\Tbf,K_\Gbf(w)_p)\geq cp^{\ord_p(w,K_\Wbf)}$ for $p\in\delta(\Tbf,K_\Gbf(w))$. 

For $p\in\Delta(\Tbf,K_\Gbf(w))$ such that $K^\maxx_{\Tbf,p}\supsetneq K_{\Tbf,p}=K_\Tbf(w)_p$, we still have $I(\Tbf,K_\Gbf(w)_p)\geq cp$ by \cite{ullmo yafaev} ?.? .
 
\end{corollary}

As we have mentioned in \ref{description of subdata}, for a subdatum $(\Pbf',Y')=\Wbf'\rtimes(w\Gbf'w^\inv,wX')$, the choice of $w$ is unique up to translation by $\Wbf'(\Qbb)$. In this case, we have:

\begin{corollary}
Let $M'$ be a special subvariety defined by a subdatum $(\Pbf',Y';Y'^+)=\Wbf'\rtimes(w\Gbf'w^\inv,wX';wX'^+)$. Then the inferium $$\inf_{w'\in\Wbf'(\Qbb)w}\prod_{p\in\Delta(\Tbf,K_\Gbf(w))}\maxx\{1,bI(\Tbf,K_\Gbf(w)_p)\}$$ is reached at some $w'\in\Wbf'(\Qbb)w$.

\end{corollary}

\begin{proof}

We first note that the representation of $\Tbf$ on $\Vbf$ (resp. on $\Ubf$) does not admit any trivial subrepresentation. Otherwise we have some $\Qbb$-subspace $\Vbf'\subset\Vbf$ stabilized by $\Tbf$ and by $\Gbf'^\der$ because they commute with each other, hence

Write $w=(u,v)$ and $w'=(u',v')(u,v)=(u+u'+\psi(v',v),v'+v)$ for $(u',v')\in\Wbf'(\Qbb)$, and $\ord_p(u',K_\Ubf)$ resp. $\ord_p(v',K_\Vbf)$ for the $p$-order of $u'$ \wrt $K_\Ubf$ resp. of $v'$ \wrt $K_\Vbf$.

We have $I(\Tbf,K_\Gbf(w')_p)\geq cp^m$ with $m={\maxx\{\ord_p(u',\ord_p(u',K_\Ubf)), \ord_p(v',K_\Vbf)\}}$ for $p\in\delta(\Tbf,K_\Gbf(w))$.  
$K_{\Tbf,p}$ is a \cosg of $\Tbf(\Qbbp)$ stabilizing $K_{\Ubf,p}$ and $K{\Vbf,p}$, hence $K_\Tbf(w')_p\subsetneq K_{\Tbf,p}$ when either $\ord_p(u',K_\Ubf)$ or $\ord_p(v',K_\Vbf)$ are large. Combining with the estimation in \ref{unipotent defects}, we see that the inferium is reached for some $w'$ such that $\ord_p(w',K_\Wbf)$ is small. \end{proof}

For convenience we introduce the following:

\begin{definition}

Let $M$ be a mixed Shimura variety defined by $(\Pbf,Y)=\Wbf\rtimes(\Gbf,X)$ at some level $K$ of fine product type. For $M'$ a special subvariety defined by $(\Pbf',Y';Y'^+)=\Wbf'\rtimes(w\Gbf'w^\inv,wX';wX'^+)$, we define the test invariant of $M'$ in $M$ to be $$\tau_M(M'):=D(\Tbf)\min_{w'\in\Wbf'(\Qbb)w}\prod_{p\in\Delta(\Tbf,K_\Gbf(w))}\maxx\{1,b\cdot I(\Tbf,K_\Gbf(w)_p)\}$$ where $\Tbf$ is the connected center of $\Gbf'$, $D(\Tbf)$ is the absolute discriminant of the splitting field of $\Tbf$, and the minimum makes sense by the corollary above.

It is actually independent of the choice of subdata defining $M'$: by \ref{independence of conjugation}, if we pass to a second defining subdatum $(\Pbf'',Y'';Y''^+)$, then its image under the natural projection is a pure subdatum $(\Gbf'',X'';X''^+)$ of $(\Gbf,X;X^+)$ with $\Gbf''=\gamma\Gbf'\gamma^\inv$ for some $\gamma\in\Gamma_\Gbf$, and its connected center is $\gamma\Tbf\gamma^\inv$, hence the absolute discriminant remains unchanged; the element $w$ could be replaced by a $\Gamma_\Wbf$-translation, which again leaves the set $\Delta(\Tbf,K_\Gbf(w))$ and the quantities $I(\Tbf, K_\Gbf(w)_p)$  unchanged.

In particular, when $M'$ is a pure special subvariety, then it is defined by some pure subdatum $(w\Gbf'w^\inv,wX';wX'^+)$ with $w$ unique up to translation by $\Gamma_\Wbf$. Different choices of $w$ gives the same value of the test invariant, and we remove the minimum in this case.

\end{definition}

We can thus transform the bounded equidistribution in Section 3 into:

\begin{proposition}[bounded test invariants]\label{bounded test invariants} Assume GRH for CM fields. Let $M$ be a mixed Shimura variety defined by $(\Pbf,Y)=\Wbf\rtimes(\Gbf,X)$ at some level $K$ of fine product type, with $E$ its reflex field. Then a sequence $(M_n)$ of special subvarieties is bounded in the sense of \ref{bounded sequence} if and only if its sequence of associated sequence of test invariants $(\tau_M(M_n))$ is bounded, i.e. $t(M_n)\leq C$ for all $n$ with $C\in\Rbb_{>0}$ some constant.

\end{proposition}

\begin{proof}

When the sequence is bounded by some $B=\{(\Tbf_1,w_1),\cdots,(\Tbf_r,w_r)\}$ then only finitely many values appear as test invariants of the sequence.

Conversely, assume that we are given a sequence of special subvarieties with test invariants uniformly bounded by some $C>0$. The natural projection $M=M_K(\Pbf,Y)\ra S=M_{K_\Gbf}(\Gbf,X)$ sends $(M_n)$ to a sequence of pure special subvarieties $(S_n)$ in $S$. If $M_n$ is $(\Tbf_n,w_n)$-special, defined by some subdatum $(\Pbf_n,Y_n;Y_n^+)=\Wbf_n\rtimes(w_n\Gbf_nw_n^\inv,w_nX_n;w_nX_n^+)$ with $\Tbf_n$ the connected center of $\Gbf_n$ and $w_n$ chosen so that the minimum in the definition of test invariants of $M_n$ is reached at $w_n$.

Then $S_n$ is a $\Tbf_n$-special subvariety of $S$. From the definition of test invariants we have $\tau_S(S_n)\leq\tau_M(M_n)$ because the two invariants involve the same $\Qbb$-torus $\Tbf_n$, and for the sets of primes of defects we have $\Delta(\Tbf_n,K_\Gbf)\subset\Delta(\Tbf_n,K_\Gbf(w_n))$. Now that $(S_n)$ is a sequence with bounded test invariants, we may apply \cite{ullmo yafaev} ?.? under GRH for CM fields, which implies the existence of a finite set of $\Qbb$-tori $\{\Cbf_1,\cdots,\Cbf_r\}$ in $\Gbf$ such that each $S_n$ is $\Cbf_i$-special  for some $i$ (and it is clear that $\Tbf_n$ is conjugate to $\Cbf_i$ by some $\gamma_n\in\Gamma_\Gbf$). We may thus assume that $\{\Tbf_n:n\in\Nbb\}=\{\Cbf_1,\cdots,\Cbf_r\}$.

Therefore only finitely many values arise as $D(\Tbf_n)$ in the test invariants $\tau_M(M_n)$, and by the assumption we see that the sequence $$\prod_{p\in\Delta(\Tbf_n,K_\Gbf(w_n))}I(\Tbf_n,K_\Gbf(w_n)_p)$$ is also uniformly bounded, hence by \ref{torsion order} and \ref{uniform defects} the classes of $w_n$'s modulo $\Gamma_\Wbf$ is finite, which means that the sequence $(M_n)$ is bounded.
\end{proof}

\section{bounded sequence and bounded Galois orbits}

Let $M$ be a mixed Shimura variety defined by $(\Pbf,Y;Y^+)=\Wbf\rtimes(\Gbf,X)$ at level $K=K_\Wbf\rtimes K_\Gbf$. If $M'$ is a special subvariety defined by $(\Pbf',Y';Y'^+)=\Wbf'\rtimes(\Gbf'X';X'^+)$, then it contains  the maximal special subvariety $S'$ defined by $(\Gbf',X';X'^+)$, which is a section to the natural projection $M'\ra S'\subset S$, and they have the same field of definition. In particular, the Galois conjugates of $M'$ are in natural bijection with those of $S'$, and in this case we have $\tau_M(M')=\tau_S(S')$, as $M'$ is $(\Tbf,0)$-special, $\Tbf$ being the connected center of $\Gbf'$.

In general we do not have an explicit way to describe the Galois conjugates of $M'$ defined by $(\Pbf',Y';Y'^+)=\Wbf'\rtimes(w\Gbf'w^\inv, wX';wX'^+)$ by the conjugates of some maximal pure special subvariety of it, unless we know a priori that $K_\Gbf=K_\Gbf(w)$. To remedy this we have the following two potential estimate:
\begin{proposition}
Assume GRH for CM fields.
Let $M$ be a mixed Shimura variety defined by $(\Pbf,Y)=\Wbf\rtimes(\Gbf,X)$ at some level $K=K_\Wbf\rtimes K_\Gbf$ of fine product type, with $E$ its reflex field. Let $(M_n)$ be a sequence of special subvarieties defined by $(\Pbf_n,Y_n;Y_n^+)=\Wbf_n\rtimes(w_n\Gbf_nw_n^\inv,w_nX_n;w_nX_n^+)$. If the sequence of test invariants $(\tau_M(M_n))$ is bounded, then there exists some \cosg $K'_\Gbf\subset K_\Gbf$ of fine product type such that when we pass to $M'=M_{K'}(\Pbf,Y)$ for the level $K'=K_\Wbf\rtimes K'_\Gbf$, the sequence $(M_n')$ with $M_n'$ defined by $(\Pbf_n,Y_n;Y_n^+)$ is of bounded Galois orbits, satisfying $$\#\Gal(\Qac/E)\cdot M_n'\geq c_ND_N(\Tbf_n)\prod_{p\in\Delta(\Tbf_n,K'_\Gbf(w_n'))}\maxx\{1,I(\Tbf_n,K'_\Gbf(w'_n)_p)$$ where $\Tbf_n$ is the connected center of $\Gbf_n$ and $w'_n\in\Wbf_n(\Qbb)w_n$.

\end{proposition}

\begin{proof}

By \ref{bounded test invariants}, the sequence $(M_n)$ is bounded, namely we can choose the defining subdata to be $(\Pbf_n,Y_n;Y_n^+)=\Wbf_n\rtimes(w_n\Gbf_nw_n^\inv,w_nX_n;w_nX_n^+)$, which are bounded by some finite set $B=\{(\Tbf_\alpha,w_\alpha):\alpha\in A\}$: $\Gbf_n$ is of connected center $\Tbf_\alpha$ and $w_n=w_\alpha$ for some $\alpha\in A$ depending on $n$. 

We thus take $K'_\Gbf=\bigcap_{\alpha\in A}K_\Gbf(w_\alpha)$, and consider the mixed Shimura variety $M'=M_{K'}(\Pbf,Y)$ with $K'=K_\Wbf\rtimes K'_\Gbf$. Now that $K'_\Gbf=K'_\Gbf(w_\alpha)$ for all $\alpha\in A$, we have $$K_\Wbf\rtimes K'_\Gbf=w_\alpha(K_\Wbf\rtimes K_\Gbf')w_\alpha^\inv=K_\Wbf\rtimes w_\alpha K'_\Gbf w_\alpha^\inv, \forall \alpha\in A$$ and $K'\cap w_\alpha\Gbf w_\alpha^\inv(\adele)=w_\alpha K'_\Gbf w_\alpha^\inv$, $\forall \alpha\in A$. In particular, the natural projection $\pi:M'=M_{K'}(\Pbf,Y)\ra S'=M_{K'_\Gbf}(\Gbf,X)$ has more pure sections than the one given by $(\Gbf,X)\mono(\Pbf,Y)$: for each $\alpha\in A$ we have $(\Gbf,X)\isom(w_\alpha\Gbf w_\alpha^\inv, w_\alpha X)\subset(\Pbf,Y)$, and the pure section it defines is $$S'(w_\alpha):=M_{w_\alpha K'_\Gbf w_\alpha}(w_\alpha\Gbf w_\alpha, w_\alpha X)\mono M'$$
which is isomorphic to $S'(0):=S'$ using the Hecke translate by $w_\alpha$, i.e. $M'\ra M'$, $[x,aK']\mapsto[x,aw_\alpha K']$ because $w_\alpha K' w_\alpha^\inv= K'$.

Therefore the Galois conjugates of pure special subvarieties in $S'(0)$ and in $S'(w_\alpha)$ are the same using the Hecke translate, and the Galois orbits of $M'_n$ is in bijection with the conjugates of its pure section $M'_n\cap S'(w_\alpha)$, as long as the original $M_n$ is $(\Tbf_\alpha,w_\alpha)$-special.
\end{proof}

The propositions above justify our use of test invariants as a substitute of the lower bound for the Galois orbit of a general special subvariety: it is ''potentially'' the correct one when we work with any bounded sequence of special subvarieties.\bigskip

We also mention the following fact, as a complement to the notion of bounded sequences:

\begin{lemma}[upper bound]\label{upper bound}

Let $M$ be a mixed Shimura variety defined by $(\Pbf,Y)=\Wbf\rtimes(\Gbf,X)$ with reflex field $E$ at some level $K=K_\Wbf\rtimes K_\Gbf$ of fine product type. Let $M'$ be a $\TW$-special subvariety of  $M$. Then we have an upper bound $$\#\Gal(\Qac/E)\cdot M'\leq c_0\cdot C(\Tbf,K_\Gbf)\ord(w,K_\Wbf)^d$$ where \begin{itemize}
\item $c_0>0$ is some constant that only depends on $\dim\Gbf$;

\item $C(\Tbf,K_\Gbf)$ is the class number $\#\Tbf(\adele)/\Tbf(\Qbb)K_\Tbf$;

\item $\ord(w,K_\Wbf)$ is the order of the class $w$ in the sense of \label{order}, and $d$ is the  square of the dimension of $\Wbf$.

\end{itemize}
\end{lemma}

In particular, a sequence of special subvarieties bounded by some finite set $B=\{\TW\}$ is of uniformly bounded Galois orbits.

\begin{proof}

We first consider the case when $w=0$, which is the same as the case of a $\Tbf$-special pure subvariety $S'$ in a given pure Shimura variety $S=M_K(\Gbf,X)$. It suffices to consider the $\Gal(\Qac/E')$-orbit of $S'$ in $S$, with $E'$ the reflex field of the subdatum defining $S'$, because $[E':E]$ is bounded by some constant that only depends on  $\dim\Gbf$. 

The size of $\Gal(\Qac/E')\cdot S'$ is the size of the image of the reciprocity map describing the Galois action permuting connected components $\rec_{(\Gbf',X')}:\Gal(\Qac/E')\ra\pibar(\Gbf')/K_{\Gbf'}$, which is reduced , up to some absolute constant that only depends on $\dim\Gbf$, to the image of $\rec_{(\Gbf',X')}^N:\Gal(\Qac/E')\ra \Tbf(\adele)/\Tbf(\Qbb)K_\Tbf$, $a\mapsto (\rec_{(\Gbf',X')}(a))^N$, $N\in\Nbb$ being some absolute constant. Hence the image is bounded by the class number $\Tbf(\adele)/\Tbf(\Qbb)K_\Tbf$ up to some constant $c_0$ that only depends on $\dim\Gbf$.

When $w\neq 0$, it suffices to shrink $K_\Gbf$ to $K_\Gbf(w)$ and replace $C(\Tbf,K_\Gbf)$ by $C(\Tbf,K_\Gbf(w))$. But $$[K_\Tbf:K_\Tbf(w)]\leq[K_\Gbf:K_\Gbf(w)]\leq\#(\Aut(K_\Ubf[w]/K_\Ubf)\times\Aut(K_\Vbf[w]/K_\Vbf))$$ which is bounded by $\ord(w,K_\Wbf)^d$ as is desired.
\end{proof}






\end{document}